\documentclass[11pt]{amsart}

\usepackage{amssymb,verbatim}
\usepackage{amsmath,amsfonts}
\usepackage[mathscr]{euscript} % make a nice \Net
\usepackage{amsthm}
\usepackage{url}
\usepackage{graphicx} %to include figure
\usepackage{enumitem} %to use nice enumeration with changed indent
\usepackage{hyperref} % to make links for lemmas and theorems
\hypersetup{colorlinks} %so the links look nicer
\usepackage{bbm} % for the characteristic function 1
\usepackage{bm} % for bold greek letter
\usepackage{mathrsfs} %to use \mathscr
\usepackage{cancel} %to make sout to cross things out

\usepackage[colorinlistoftodos]{todonotes}

\RequirePackage{cleveref}
\usepackage{hypcap}
\hypersetup{colorlinks=true, citecolor=darkblue, linkcolor=darkblue}
\definecolor{darkblue}{rgb}{0.0,0,0.7}
\newcommand{\darkblue}{\color{darkblue}}

\definecolor{darkred}{rgb}{0.68,0,0}

\definecolor{darkgreen}{rgb}{0,.38,0}
\newcommand{\darkgreen}{\color{darkgreen}}

\newcommand{\defn}[1]{\emph{\darkblue #1}}

\newcommand{\defnb}[1]{\emph{\darkblue #1}}
\newcommand{\defng}[1]{\emph{\darkgreen #1}}

\setlist[enumerate]{
	label=\textnormal{({\roman*})},
	ref={\roman*}}

\makeatletter
\def\th@plain{%
	\thm@notefont{}% same as heading font
	\itshape % body font
}
\def\th@definition{%
	\thm@notefont{}% same as heading font
	\normalfont % body font
}
\makeatother

\newtheorem{thm}{Theorem}[section]
\newtheorem{lemma}[thm]{Lemma}

\newtheorem*{claim*}{Claim}
\newtheorem{cor}[thm]{Corollary}
\newtheorem{prop}[thm]{Proposition}
\newtheorem{conj}[thm]{Conjecture}
\newtheorem{conv}[thm]{Convention}

\theoremstyle{definition}

\newtheorem{rem}[thm]{Remark}

\numberwithin{figure}{section}
\numberwithin{equation}{section}

\def\nn{\mathbb N}

\def\qqq{\mathbb Q}

\def\pp{\mathbb P}

\def\ga{\gamma}

\def\al{\alpha}
\def\be{\beta}

\def\ve{\varepsilon}

\def\cT{\mathcal T}

\def\ssu{\subset}

\def\<{\langle}
\def\>{\rangle}

\def\width{\text{{\rm width}}}

\def\0{{\mathbf 0}}

\def\.{\hskip.06cm}
\def\ts{\hskip.03cm}

\def\di{{\small{\ts\diamond\ts}}}

\def\ag{\rho}
\def\agr{\ag}

\def\aS{\textrm{S}}

\def\as{\textrm{s}}
\def\asr{\textrm{\em s}}

\def\aG{\textrm{G}}
\def\aGr{\textrm{\em G}}
\def\agg{\textrm{g}}
\def\aggr{\textrm{\em g}}

\def\aR{\textrm{R}}
\def\aRr{\textrm{\em R}}
\def\ar{\textrm{r}}
\def\arr{\textrm{\em r}}

\def\.{\hskip.06cm}
\def\ts{\hskip.03cm}

\def\nin{\noindent}

\DeclareMathOperator{\Ec}{\mathcal{E}} %The set of all linear extensions
\makeatletter
\newcommand{\ogeneric}[2][0.7]{%
	\vphantom{\oplus}\mathpalette\o@generic{{#1}{#2}}%
}
\newcommand{\o@generic}[2]{\o@@generic#1#2}
\newcommand{\o@@generic}[3]{%
	\begingroup
	\sbox\z@{$\m@th#1\oplus$}%
	\dimen@=\dimexpr\ht\z@+\dp\z@\relax
	\savebox\tw@[\totalheight]{$\m@th#1\bigcirc$}%
	\makebox[\wd\z@]{%
		\ooalign{%
			$#1\vcenter{\hbox{\resizebox{\dimen@}{!}{\usebox\tw@}}}$\cr
			\hidewidth
			$#1\vcenter{\hbox{\resizebox{#2\dimen@}{!}{$#1\vphantom{\oplus}{#3}$}}}$%
			\hidewidth
			\cr
		}%
	}%
	\endgroup
}
\makeatother

\newcommand{\ole}{\mathrel{\ogeneric{<}}}
\title{Linear extensions and continued fractions}
\date{\today}

 \author{Swee Hong Chan}
 \address[Swee Hong Chan]{Department of Mathematics, Rutgers University,  Piscatway, NJ 08854.}
 \email{\texttt{sc2518@rutgers.edu}}

 \author[\ts Igor Pak]{Igor Pak}
 \address[Igor Pak]{Department of Mathematics, UCLA,  Los Angeles, CA 90095.}
 \email{\texttt{pak@math.ucla.edu}}

\begin{document}

\begin{abstract}
We introduce several new constructions of finite posets with the number
of linear extensions given by generalized continued fractions.
We apply our results to the problem of the minimum number of
elements needed for a poset with a given number of linear
extensions.
\end{abstract}
	
\maketitle
	
\vskip-.5cm

% \newpage

\section{Introduction} \label{s:intro}
%
%A bridge between areas is a wonderful thing to have.

\subsection{Foreword} \label{ss:intro-for}
Continued fractions go back to antiquity \cite{Bre91} and
are surprisingly versatile.  They appear across mathematics,
from number theory \cite{BPSZ14,RS92} to analysis \cite{JT80,Khi97}, from
cluster algebras \cite{CS18} to discrete geometry \cite{Kar13}
to signal processing \cite{Sau21}.
In combinatorics, they famously enumerate partitions \cite{AB05},
lattice paths \cite{Fla80} (see also \cite{FS,GJ83,PSZ23}), permutations
\cite{Eli17,SZ22}, and perfect matchings \cite{Vie85} (see also \cite{Sch19,Spi21}).

Curiously, the applications go in both directions:
the asymptotics of combinatorial sequences can be derived from analytic
properties of continued functions, while combinatorial interpretations
imply positivity properties.  This paper explores connections between
linear extensions of finite posets and continued fractions, and their
asymptotic applications to counting.

Note that we utilize standard terminologies in order theory and continued fraction theory, and we include detailed definitions in Section~\ref{s:def} for the reader's convenience and clarity.

\smallskip

\subsection{Linear extensions} \label{ss:intro-LE}
Let \. $P=(X,\prec)$ \. be a poset with \. $|X|=n$ \. elements.
Denote \. $[n]:=\{1,\ldots,n\}$.
A \defn{linear extension} of $P$ is a bijection \. $f: X \to [n]$,
such that
\. $f(x) < f(y)$ \. for all \. $x \prec y$.
Let \ts $\Ec(P)$ \ts be the set of linear extensions of $P$,
and denote \. $e(P):=|\Ec(P)|$.  Clearly, \ts $1\le e(P) \le n!$ \.
See \cite{CP23-survey} for a detailed recent survey.

Denote by \ts $\mu(n)$ \ts the minimum number of elements in
a poset with \ts $n$ \ts linear extensions. See
\cite[\href{http://oeis.org/A160371}{A160371}]{OEIS} for the
numerical data (see also
\cite[\href{http://oeis.org/A281723}{A281723}]{OEIS}).
For example, $\mu(5) = 4$ \ts since $e(Z_4)=5$, where $Z_4$ is a
\defng{zigzag poset} \ts on $4$ elements (with an $N$-shaped comparability graph).

The asymptotics of \ts $\{\mu(n)\}$ \ts remains an important open problem.  Clearly,
\ts $\mu(n) \le n$ \ts since for the parallel sum of chains we have:
\ts $e(C_{n-1} \oplus C_1)=n$. In a different direction, \ts $\mu(n) = \Omega(\log n/\log \log n)$ \ts since
\ts $e(P) \le n!$ \..  The first nontrivial upper bound \ts $\mu(n) = O(\sqrt{n})$ \ts
was found by Tenner~\cite{Ten}.  Most recently, this bound was greatly improved:

\begin{thm}[{\rm Kravitz--Sah \cite[Thm~1.1]{KS21}}{}] \label{t:KS-main}
We have: \.  $\mu(n) = O(\log n \ts \log \log n)$.
\end{thm}

The authors use a simple but surprising connection to continued fractions,
the starting point of this paper (see below).  They state the following:

\begin{conj}[{\rm \cite[Conj.~7.3]{KS21}}{}]  \label{conj:KS-main}
We have: \.  $\mu(n) = O(\log n)$.
\end{conj}

% While it would be excellent to resolve the conjecture
In this paper, we are mostly interested in the combinatorial aspects of
the connection between linear extensions and continued fractions,
suggesting new technical tools towards the conjecture.

\smallskip

\subsection{Simple continued fractions} \label{ss:intro-CF}
Let \. $\nn:=\{0,1,2,\ldots\}$ \. and \. $\pp:=\{1,2,\ldots\}$.
A \defnb{simple continued fraction} \ts (CF) \ts is  defined as follows:
\begin{equation}\label{eq:CF-def}
[b_0,b_1,b_2,\ldots, b_m]  \ := \  b_0 \, + \, \cfrac{1}{b_1 \, + \, \cfrac{1}{b_2\, + \, \cfrac{1}{\ \ddots \, + \, \frac{1}{b_m}} }} \ ,
\end{equation}

\smallskip

\nin
where integers \ts $b_0 \ge 0$, \ts $b_1,\ldots, b_{m-1} \ge 1$, and \ts $b_m \ge 2$ \ts for \ts $m\ge 1$.
Integers \ts $b_i$ \ts are called \defn{quotients}.  The sum of these quotients \.
$\aS(b_0,\ldots, b_m):= b_0 + \ldots + b_m$ \. is called the \defn{weight} \ts of \ts
$[b_0,\ldots, b_m]$.  Recall that for every \ts $\al \in \qqq_{\ge 0}$ \ts there is a
unique simple continued fraction \ts $[b_0,b_1,b_2,\ldots, b_m]=\al$,
and we write \. $\as(\al):=\aS(b_0,b_1,b_2,\ldots, b_m)$ \. in this case. Note that \.
$\as(\al)=\as\big(\al^{-1}\big)$.

In the terminology of \cite{YK75} (see also  \cite[$\S$4.5.3]{Knuth98}),  the weight
\ts $\as\big(\tfrac{c}{d}\big)$ \ts is the number of steps of the \defn{subtraction algorithm},
the original (classical) version of the Euclidean algorithm for finding the greatest
common divisor that uses only subtractions instead of divisions.
The following result is the key to the
proof of Theorem~\ref{t:KS-main}.

\begin{thm}[{\rm Larcher \cite{Lar86}, see also \cite[Thm~1.2]{KS21}}{}] \label{t:Larcher}
For every integer \ts $d\ge 1$, there exists an integer \ts $1\le c < d$, \. $\gcd(c,d)=1$,
such that % \. $c/d= [0,b_1,\ldots,b_m]$ \. and
\begin{equation}\label{eq:Larcher}
\asr\big(\tfrac{c}{d}\big) \, \le \, C \. \frac{d}{\phi(d)} \, \log d \, \log \log d\ts,
\end{equation}
where \ts $\phi(n)$ \ts is Euler's totient
function, and \ts $C>0$ \ts is a universal constant.
\end{thm}

See $\S$\ref{ss:finrem-hist} for more on the theorem.
Now, Kravitz and Sah observed that Conjecture~\ref{conj:KS-main}
follows from the following conjectural extension of Theorem~\ref{t:Larcher}.

\begin{conj}[{\rm \cite[Conj.~7.2]{KS21}}{}]  \label{conj:KS-asy}
For every prime \ts $d$, there is an integer \ts $1\le c < d$,
such that % \. $c/d= [0,b_1,\ldots,b_m]$ \. and
\begin{equation}\label{eq:Larcher}
\asr\big(\tfrac{c}{d}\big) \, \le \, C \.  \log d\ts,
\end{equation}
where \ts $C>0$ \ts is a universal constant.
\end{conj}

Note that in a CF \eqref{eq:CF-def} for \ts $\frac{c}{d}\ts$, the number
of quotients is  \ts $m=O(\log d)$.  Thus, Conjecture~\ref{conj:KS-asy} follows from the
celebrated \defn{Zaremba's conjecture} (see also~$\S$\ref{ss:finrem-Zar}):

\begin{conj}[{\rm Zaremba \cite[p.~76]{Zar72}}{}] \label{conj:Zaremba}
For every integer \ts $d\ge 1$, there is an integer \ts $1\le c < d$,
such that \. $c/d= [0,b_1,\ldots,b_m]$ \. and \. $b_1,\ldots,b_m\le A$,
where \ts $A>0$ \ts is a universal constant.
\end{conj}

\smallskip

\subsection{From continued fractions to linear extensions} \label{ss:intro-LE-CF}
In poset \ts $P=(X,\prec)$, an \defn{antichain} \ts is a subset of pairwise independent
elements.  The \defn{width} \ts of a poset is the size of the maximal antichain.
An element $x\in X$ is \defn{minimal}, if for every \ts $y\in X$ \ts we have either
\ts $x\preccurlyeq y$ \ts or \ts $x \.\|\.y$.
Denote by \ts $\min(P)$ \ts the set of all minimal elements in~$P$.
Denote by \ts $P-x$ \ts the poset obtained by removing the element $x$.

\begin{thm}[{\rm see \cite[Prop.~4.1]{KS21}}{}] \label{t:KS-CF}
For all integers \ts $1\le c< d$ \ts with \ts $\gcd(c,d)=1$,
%and a simple continued fraction \. $[0,b_1,\ldots,b_m]=c/d$,
there is a poset \ts $P=(X,\prec)$ \ts of width two,
such that \ts $|X|=\asr\big(\tfrac{c}{d}\big)$, \ts $e(P)=d$ \ts and \ts $e(P-x)=c$ \ts
for some minimal element \ts $x\in \min(P)$.
\end{thm}

The proof of the theorem uses two simple transformations of posets \ts
$(P,x) \to (P',x')$ \ts and \ts $(P'',x'')$, such that for \. $e(P)=d$, $e(P-x)=c$ \.
the new posets satisfy \. $e(P') = e(P'')=c+d$, \. $e(P'-x')=c$, \. $e(P''-x'')=d-c$.
In Section~\ref{s:rec} we modify and generalize this construction.

Before we proceed to generalizations, consider
$$\cT(k) \,:= \, \big\{e(P) \, : \, P=(X,\prec), \, |X| \le k \big\},
$$
so that \. $\mu(n) =\min\{ \ts k \ts : \ts n \in  \cT(k) \ts \}$.  Open Problems~7.5
and~7.6 in~\cite{KS21} ask about the asymptotics of \ts $|\cT(k)|$,
and of the largest \ts
$L=L_c(k)$ \ts such that \ts $\big|\cT(k)\cap \{1,\ldots,L\}\big| > c\ts L$.
We have the following direct application of Theorem~\ref{t:KS-CF}
(not noticed in~\cite{KS21}),
which gives partial answers to both open problems:
%\com{SH}{Corollary 1.7 also answers Conjecture 7.3 and 7.4 in KS paper :).}

\begin{cor}\label{c:BK}  We have: \. $|\cT(k)| = \exp \Omega(k)$.
Moreover, there is a constant \ts $C>1$, such that
\begin{equation}\label{eq:BK-cor}
\frac{1}{C^k}\.\big|\ts \cT(k) \cap \big\{1,2,\ldots, \lfloor C^k\rfloor\big\} \ts\big| \, \to \, 1 \quad \text{as} \ \ k \to \infty.
\end{equation}
\end{cor}

%\com{SH}{In KS, they define $LE(k)$ to be \. $ \big\{e(P) \, : \, P=(X,\prec), \, |X| = k \big\}$\. (note the equal sign), which is slightly different from $\cT(k)$, so it becomes a bit arguable if Corollary 1.7 solves Problem 7.5 and 7.6.
%}

\begin{proof}  Recall the following remarkable result of Bourgain and Kontorovich
\cite{BK14} (see also~$\S$\ref{ss:finrem-Zar}), giving an asymptotic version of
Zaremba's Conjecture~\ref{conj:Zaremba}:
\. $\ga(n)\to 1$ \. as \. $n\to \infty$, where $\ga(n)$ \ts denotes the proportion
of $d \in \{1,\ldots,n\}$, such that \ts $c/d$ \ts has all quotients \ts $\le 50$ \ts
for some \ts $1\le c <d$, \ts $\gcd(c,d)=1$.  Since \ts $\as\big(\frac{c}{d}\big) = O(\log d)$ \ts
for such fractions, by Theorem~\ref{t:KS-CF} we obtain the result.
\end{proof}

% \begin{proof}[Sketch of proof of Theorem~\ref{t:KS-main}]
% Note that the RHS of \eqref{eq:Larcher}  is \. $O(\log d \, \log \log d)$ \ts for
% all prime~$d$.  To obtain the theorem, for every prime \. $d \ts | \ts n$,
% apply Theorem~\ref{t:Larcher} and then Theorem~\ref{t:KS-CF} to obtain a poset \ts
% $P_d$ \ts with $e(P_d)=d$.   Taking a linear sum of posets \ts $P_d$ \ts gives
% a poset \ts $P=(X,\prec)$ \ts on \ts $|X|=O(\log n \ts \log \log n)$ \ts elements, and
% with \ts $e(P)=n$ \ts linear extensions.  See \cite[$\S4.2$]{KS21} for details.
% \end{proof}

\smallskip

\subsection{Relative version} \label{ss:intro-rel}
Let \ts $P=(X,\prec)$ \ts and let \ts $x\in X$.  Following \cite{CP},
consider the \defn{relative number of linear extensions}:
\[ \ag(P,x)  \, := \,  \frac{e(P)}{e(P-x)}\..
\]
It follows from Theorem~\ref{t:KS-CF}, that every rational number
\ts $\al \ge 1$ \ts is equal to \ts $\ag(P,x)$ \ts for some poset~$P$
and element \ts $x\in X$.

For \. $d\ge c \ge 1$,
let \ts $\nu(c,d)$ \ts denote the minimal number of elements in a poset
\ts $P=(X,\prec)$, such that \ts $\rho(P,x) = \frac{d}{c}$ \ts for some \ts $x\in X$.
The following upper bound can be viewed as a relative version of Theorem~\ref{t:KS-main}.

\smallskip

\begin{thm}\label{t:CP-relative}
For all \. $d \ge 3\ts c$, we have:
\begin{equation}\label{eq:CP-rel}
\nu(c,d) \, \le \, \frac{d}{c} \. + \. O(\log d \. \log \log d).
\end{equation}
\end{thm}

\smallskip

In \cite[Prop.~8.8]{CP}, we showed an asymptotically matching lower bound:
\begin{equation}\label{eq:CP-rel-lower}
\nu(c,d) \, \geq \, \frac{d}{c}\..
\end{equation}
The motivation  of Theorem~\ref{t:CP-relative} comes from the approach
in \cite{CP23a}, where we studied relative versions of several
counting functions (domino tilings, spanning trees, etc.)
The proof of Theorem~\ref{t:CP-relative} is based
on the approach in \cite[$\S$8.2]{CP}.
It would be interesting to see if the condition \ts $d\ge 3 \ts c$ \ts
can be weakened to \ts $d \ge (1+\ve) \ts c$ \ts or even
dropped.  Additionally, by analogy with the Kravitz--Sah
Conjecture~\ref{conj:KS-main}, we conjecture that \eqref{eq:CP-rel}
can be improved to
\begin{equation}\label{eq:rel-conj}
	\nu(c,d) \, \le \, \tfrac{d}{c} \. + \. O(\log d).
\end{equation}
In a different direction, one can ask about the smallest size poset
with \ts $e(P)=d$ \ts and \ts $e(P-x)=c$, since the construction
in the proof can result in an integer multiple of both.

The key part of the proof of Theorem~\ref{t:CP-relative} is the following tail estimate for
the weight of random continued fractions:

\begin{thm}[{\rm Rukavishnikova \cite{Ruka}}{}]\label{t:Ruka}
There is a universal constant \ts $C>0$, such that
\begin{equation}\label{eq:Ruka}
    \frac{1}{d} \, {\ }\#\left\{ c \in [d] \. : \. \left|\asr\big(\tfrac{c}{d}\big) \. - \. \tfrac{12}{\pi^2} \. \log d \. \log\log d\right| \,
    > \,  (\log d)(\log \log d)^{2/3}\right\} \ < \frac{C}{(\log \log d)^{1/3}}\,.
%  \quad \text{as \ \  $n\to \infty$}.
\end{equation}
\end{thm}

\smallskip

\nin
Here we are stating a special case of the main theorem in \cite{Ruka}
which suffices for our purposes.
%\newpage

\smallskip

\subsection{Generalized continued fractions} \label{ss:intro-GCF}
Let \ts $m \geq 0$, \ts $a_1,\ldots, a_m \in \pp$, \ts $b_0,\ldots, b_m \in \pp$.
A \defnb{generalized continued fraction} \ts (GCF) \ts is defined as
\begin{equation}\label{eq:GCF-def}
[a_1,\ldots, a_m; b_0,\ldots, b_m]  \ := \  b_0 \, + \, \cfrac{a_1}{b_1 \, + \, \cfrac{a_2}{b_2 \, + \,
\cfrac{a_3}{\ \ddots \, + \, \frac{a_{m}}{b_m}} }} \ .
\end{equation}
Note that when \.  $a_1=\ldots=a_m=1$ \. we get a simple continued fraction.
We define the  \defnb{weight} of GCFs as follows:
\[
\aG(a_1,\ldots,a_m \ts ; \ts b_0,\ldots, b_m)  \, := \, (b_0+\ldots +b_m) \. - \. (a_1 +\ldots + a_m) \. + \. m,
\]
and note that \. $\aG(1,\ldots,1 \ts ; \ts b_0,\ldots, b_m) = \aS(b_0,\ldots, b_m)$.
Observe that a rational number can have many presentations as a GCF, some of which
can have weight smaller than the weight of the corresponding CFs.  For example,
\[ \frac{20}{7} \ = \  2+ \cfrac{1}{1+\frac{1}{6}}  \ = \   2+ \cfrac{2}{2+\frac{1}{3}} \ , \]
so \ts $\as\big(\tfrac{20}{7}\big) = \aS(2,1,6) = 9$ \ts and \ts $\aG(2,1\ts;\ts2,2,3) = 6$.

A generalized continued fraction \eqref{eq:GCF-def} is called \defn{balanced} \ts % (BCF)
\ts if  % \ts $b_0 \ge a_1-1$, \. $b_m \ge a_m-1$,
\begin{equation}\label{eq:GCF-balanced}
	b_i \, \geq \, a_i + a_{i+1} -1 \quad \text{for all} \quad 0 \. \le \. i \. \le \. m\ts,
\end{equation}
where by convention we assume that \ts $a_0=a_{m+1}=1$.
Clearly, every simple continued fraction of \ts $\al \in \qqq_{\ge 1}$ \ts is balanced.
The following is the GCF analogue of Theorem~\ref{t:KS-CF}.

\begin{thm}\label{t:GCF}
Let \ts $m \geq 0$\ts, \ts $a_1,\ldots, a_m \in \pp$\ts, \ts $b_0,\ldots, b_m \in \pp$ \ts
be integers satisfying \eqref{eq:GCF-balanced}.
Then there exists a poset \ts $P=(X,\prec)$ \ts of width at most three,
and a minimal element \ts $x \in \min(P)$, such that \. $|X| = \aGr(a_1,\ldots,a_m \ts ; \ts b_0,\ldots,b_m)$,
and
$$
[a_1,\ldots, a_m \ts ; \ts b_0,\ldots, b_m] \, = \, \ag(P,x)\ts,
$$
where \. $[a_1,\ldots, a_m \ts ; \ts b_0,\ldots, b_m]$ \. is a balanced GCF defined in \eqref{eq:GCF-def}.
\end{thm}

For \ts $\al \in \qqq_{\ge 1}$\ts, define
$$\agg(\al) \, := \, \min \big\{\.\aG(a_1,\ldots, a_m \ts ; \ts b_0,\ldots, b_m) \, : \, [a_1,\ldots, a_m \ts ; \ts b_0,\ldots, b_m] \ts = \ts \al \. \big\},
$$
where the minimum is over all balanced GCF \eqref{eq:GCF-balanced} such that all
partial fractions \ts $\frac{C_i}{D_i}$ \ts are reduced, i.e.\ $\gcd(C_i,D_i)=1$ \ts for all \ts
$1\le i \le m$ (see the definition in~$\S$\ref{ss:def-CF}).
For example, if \ts $a_1=\ldots=a_m=r$ \ts for some integer \ts $r\ge 1$, and integers
\ts $b_1,\ldots,b_m$ are coprime to~$r$, then all partial fractions \ts $\frac{C_i}{D_i}$ \ts
are reduced.  In particular, this condition automatically holds for all simple CFs. %, but not always for GCFs.
From above, we have \. $\agg(\al) \le \as(\al)$.  Thus, the following conjecture is a
natural weakening of Conjecture~\ref{conj:KS-asy}.

\begin{conj}  \label{conj:gen-asy}
For every prime \ts $d$, there is an integer \ts $1\le c < d$,
such that % \. $c/d= [0,b_1,\ldots,b_m]$ \. and
\begin{equation}\label{eq:Larcher-gen}
\aggr\big(\tfrac{d}{c}\big) \, \le \, C \.  \log d\ts,
\end{equation}
where \ts $C>0$ \ts is a universal constant.  % Moreover, the claim holds for every integer \ts $d\ge 1$.
\end{conj}

From Theorem~\ref{t:GCF}, we have:

\begin{prop}\label{p:GCF-imply}
Conjecture~\ref{conj:gen-asy} implies Conjecture~\ref{conj:KS-main}.
\end{prop}
\smallskip

\subsection{Rational GCFs} \label{ss:intro-k-GCF}
We call a continued fraction of the form \eqref{eq:GCF-def} \defn{rational} \ts if $a_i \in \qqq_{\ge 1}$.
A rational generalized continued fraction (RGCF) is called \defn{balanced} \ts if it is of the form
% Note that \. $\aS^{(k)}(\alpha) \in \nn$.
\begin{equation}\label{eq:RGCF-def}
b_0 +  \alpha_1  + \, \cfrac{\alpha_1}{\as(\alpha_1)-1 +b_1+\alpha_2+ \, \cfrac{\al_2}{\as(\al_2)-1+ b_2+\al_3+ \, \cfrac{\al_3}{\ \ddots \, + \, \frac{\al_{m}}{\as(\al_m)-1+b_{m}}} }} \ ,
\end{equation}
where \. $\alpha_1,\ldots, \alpha_m \in  \qqq_{\ge 1}$ \. and \.
$b_0,\ldots, b_{m}  \in \nn$ \. s.t.\ $b_m\ge 1$.
We use \. $[\alpha_1,\ldots, \alpha_m \ts; \ts b_0,\ldots, b_m]$ \. to denote this RGCF.

Note that for \. $\al_1,\ldots,\al_m \in \pp$, this is a balanced GCF, since the inequalities \eqref{eq:GCF-balanced}
are automatically satisfied.  Denote by
\[  \aR(\alpha_1,\ldots, \alpha_m \ts; \ts b_0,\ldots, b_m) \, := \,  b_0 + \ldots + b_m  \. + \.
\as(\alpha_1)  + \ldots + \as(\alpha_m)
\]
the \defn{weight} \ts of \eqref{eq:RGCF-def}.  For example, take \ts $m=1$, \. $\al_1=\frac32$\ts,
\. $b_0=1$, \ts $b_1=3$.  Then
$$
\as\big(\tfrac32\big) \. = \. 3, \quad
\big[\tfrac32 \. ; \ts 1,3\big] \. = \. 1 + \tfrac32 \. + \. \frac{\tfrac32}{\as\big(\tfrac32\big)-1+3} \. = \. \tfrac{14}{5}
\quad \text{and} \quad \aR\big(\tfrac32 \. ; \ts 1,3\big) \. = \. 1 + 3 + \as\big(\tfrac32\big) \. = \. 7.
$$

The following result is a variation of Theorem~\ref{t:GCF} to RGCF:

\begin{thm}\label{t:RGCF}
Let \ts $m \geq 0$, \ts $\al_1,\ldots, \al_m \in \qqq_{\ge 1}$ \ts and \ts $b_0,\ldots, b_m \in \pp$.
Then there exists a poset \ts $P=(X,\prec)$ \ts of width at most three,
and a minimal element \ts $x \in \min(P)$, such that \. $|X| = \aRr(\al_1,\ldots,\al_m \ts ; \ts b_0,\ldots,b_m)$,
and
$$
[\al_1,\ldots,\al_m \ts ; \ts b_0,\ldots,b_m] \, = \, \ag(P,x)\ts,
$$
where \. $[\al_1,\ldots,\al_m \ts ; \ts b_0,\ldots,b_m]$ \. is a balanced
RGCF defined in \eqref{eq:RGCF-def}.
%\com{SH}{We only need width three for this result, and width four is only needed if we change from $\as(\alpha)$ to $\agg(\alpha)$. Currently the statement of the theorem is written for $\as(\alpha)$, but proof is for $\agg(\alpha)$. I am fine with either one, and I will then change it to the version you prefer.}
\end{thm}

For \ts $\be \in \qqq_{\ge 1}$\ts, define
$$\ar(\be) \, := \, \min \big\{\.\aR(\al_1,\ldots, a_m \ts ; \ts b_0,\ldots, b_m) \, : \, [\al_1,\ldots, \al_m \ts ; \ts b_0,\ldots, b_m] \ts = \ts \be \. \big\},
$$
where the minimum is over all RGCF \eqref{eq:RGCF-def} such that all
partial fractions \ts $\frac{C_i}{D_i}$ \ts are reduced
(see the definition in~$\S$\ref{ss:def-CF}).
%  : \. $\gcd(C_i,D_i)=1$ \ts for all \ts
% $1\le i \le m$ (see the notation in~$\S$\ref{ss:def-CF}).
From above, \. $\ar(\al) \le \as(\al)$.
Thus, the following conjecture is a natural weakening of both
Conjecture~\ref{conj:KS-asy} and Conjecture~\ref{conj:gen-asy}.

\begin{conj}  \label{conj:rat-asy}
For every prime \ts $d$, there is an integer \ts $1\le c < d$,
such that % \. $c/d= [0,b_1,\ldots,b_m]$ \. and
\begin{equation}\label{eq:Larcher-improve}
\arr\big(\tfrac{d}{c}\big) \, \le \, C \.  \log d\ts,
\end{equation}
where \ts $C>0$ \ts is a universal constant. % Moreover, the claim holds for every integer \ts $d\ge 1$.
\end{conj}

To motivate the conjecture, note that \ts $\ar(\be)$ \ts can be much
smaller than \ts $\as(\be)$.  Take, for example, \ts $m=1$, \ts
$\al_1 = \frac{13}{7}$ \ts and \ts $\be =\frac{173}{56}$.  We have:
$$\al_1 \. = \. 1 + \cfrac{1}{1+\frac{1}{6}} \ , \quad
\as(\al_1) \. = \, 8\,, \quad
\be \. = \. 3 + \cfrac{1}{11+\frac{1}{5}}  \. = \. 1 + \al_1 + \cfrac{\al_1}{s(\al_1)-1+1}  \. = \. [\al_1\ts;\ts 1,1]\..
$$
Thus, \. $\as(\be) = \agg(\be) = 19$ \. while \. $\ar(\be) \le \aR(\al_1\ts;\ts 1,1) = 10$ \. in this case.
Again, by Theorem~\ref{t:RGCF} we have:

\smallskip

\begin{prop} \label{p:RGCF-imply}
Conjecture~\ref{conj:rat-asy} implies Conjecture~\ref{conj:KS-main}.
\end{prop}
%\begin{ex}
%\end{ex}

\smallskip

\subsection{Paper structure} \label{ss:intro-structure}
We recall poset theoretic definitions and notation in
Section~\ref{s:def}.  Recursive constructions of posets
are studied in Section~\ref{s:rec}.  We present the proofs in
Section~\ref{s:proofs}.  We conclude with final remarks and
open problems in Section~\ref{s:finrem}.

%%%%%%%%%%%%%%%%%%%%%%%%%%%%%%%%%%%%%%%%%%%%%%%%%%%%%%%%%%%%%%%%%%%%%%%%%%%%%%%%%%%%%%%%%%%%%%%%%
%%%%%%%%%%%%%%%%%%%%%%%%%%%%%%%%%%%%%%%%%%%%%%%%%%%%%%%%%%%%%%%%%%%%%%%%%%%%%%%%%%%%%%%%%%%%%%%%%
%%%%%%%%%%%%%%%%%%%%%%%%%%%%%%%%%%%%%%%%%%%%%%%%%%%%%%%%%%%%%%%%%%%%%%%%%%%%%%%%%%%%%%%%%%%%%%%%%
%%%%%%%%%%%%%%%%%%%%%%%%%%%%%%%%%%%%%%%%%%%%%%%%%%%%%%%%%%%%%%%%%%%%%%%%%%%%%%%%%%%%%%%%%%%%%%%%%
%%%%%%%%%%%%%%%%%%%%%%%%%%%%%%%%%%%%%%%%%%%%%%%%%%%%%%%%%%%%%%%%%%%%%%%%%%%%%%%%%%%%%%%%%%%%%%%%%
%%%%%%%%%%%%%%%%%%%%%%%%%%%%%%%%%%%%%%%%%%%%%%%%%%%%%%%%%%%%%%%%%%%%%%%%%%%%%%%%%%%%%%%%%%%%%%%%%
%%%%%%%%%%%%%%%%%%%%%%%%%%%%%%%%%%%%%%%%%%%%%%%%%%%%%%%%%%%%%%%%%%%%%%%%%%%%%%%%%%%%%%%%%%%%%%%%%
%%%%%%%%%%%%%%%%%%%%%%%%%%%%%%%%%%%%%%%%%%%%%%%%%%%%%%%%%%%%%%%%%%%%%%%%%%%%%%%%%%%%%%%%%%%%%%%%%
%%%%%%%%%%%%%%%%%%%%%%%%%%%%%%%%%%%%%%%%%%%%%%%%%%%%%%%%%%%%%%%%%%%%%%%%%%%%%%%%%%%%%%%%%%%%%%%%%

\medskip

\section{Basic definitions and notation} \label{s:def}

\subsection{Posets} \label{ss:def-posets}
For a poset \ts $P=(X,\prec)$ \ts and a subset \ts $Y \ssu X$, denote
by \ts $P_Y=(Y,\prec)$ \ts a \defn{subposet} \ts of~$P$.  We use \ts
$(P-z)$ \ts to denote a subposet \ts $P_{X-z}$\ts, where \ts $z\in X$.
Element \ts $x\in X$ \ts is \defn{minimal} \ts in~$\ts P$, if there
exists no element  \ts $y \in X-x$ \ts such that \ts $y \prec x$\ts.
Denote by \ts $\min(P)$ \ts the set of minimal elements in~$P$.

In a poset \ts $P=(X,\prec)$, elements \ts $x,y\in X$ \ts are called
\defn{incomparable} if \ts $x\not\prec y$ \ts
and \ts $y \not \prec x$.  We write \. $x\parallel y$ \. in this case.
An \defn{antichain} \ts is a subset \ts $A\ssu X$ \ts of pairwise incomparable elements.
The \defn{width} of poset  \ts $P=(X,\prec)$, denoted \ts $\width(P)$,
is the size of a maximal antichain.
A \defn{chain} \ts is a subset \ts $C\ssu X$ \ts of pairwise
comparable elements.  Denote by  \ts $A_n$ \ts and \ts $C_n$ \ts the
antichain and the chain with \ts $n$ \ts elements, respectively.

A \defn{dual poset} \ts is a poset \ts $P^\ast=(X,\prec^\ast)$, where
\ts $x\prec^\ast y$ \ts if and only if \ts $y \prec x$.
A \defn{parallel sum} \ts $P \oplus Q$ \ts of posets \ts $P=(X,\prec)$ \ts
and \ts $Q=(Y,\prec')$ \. is a poset \ts $(X\cup Y,\prec^\di)$,
where the relation $\prec^\di$ coincides with $\prec$ and $\prec'$ on
$X$~and~$Y$, and \. $x\.\|\. y$ \. for all \ts $x\in X$, $y\in Y$.
A \defn{linear sum} \ts $P\ole Q$ \ts of posets \ts $P=(X,\prec)$ \ts
and \ts $Q=(Y,\prec')$ \. is a poset \ts $(X\cup Y,\prec^\di)$,
where the relation $\prec^\di$ coincides with $\prec$ and $\prec'$ on
$X$~and~$Y$, and \. $x\prec^\di y$ \. for all \ts $x\in X$, $y\in Y$.

Note that \ts $e(P^\ast)=e(P)$, \ts $e(P \ole Q)=e(P) \. e(Q)$ \ts and
\ts $e(P\oplus Q) = \binom{n+n'}{n} \. e(P)\. e(Q)$, where \ts $|X|=n$
\ts and \ts $|Y|=n'$.
We refer to \cite[Ch.~3]{Sta-EC} for an accessible introduction,
and to surveys \cite{BrW,CP23-survey,Tro} for further definitions and standard
results.

\smallskip

\subsection{Continued fractions}\label{ss:def-CF}
Consider a GCF \. $[a_1,\ldots, a_m  \ts ; \ts  b_0,\ldots, b_m]$ \. given by \eqref{eq:GCF-def}.
Recursively define \. $C_i:=C_i(a_1,\ldots,a_m \ts ; \ts b_0,\ldots, b_m)$ \. and \.
$D_i := D_i(a_1,\ldots,a_m \ts ; \ts b_0,\ldots, b_m)$,  \.  $0 \le i \le m$,
 as follows:
\begin{align*}
	&C_m \ := \  b_m \,, \qquad D_m \ := \  1,\\
	& D_i  \ := \  C_{i+1}\,, \qquad C_{i} \ := \  b_{i} \. D_i  \. + \. a_{i+1} \. D_{i+1}.
\end{align*}
It is  easy to see by induction that
\[ [a_{i+1},a_{i+2}, \ldots, a_m \ts ; \ts  b_i, b_{i+1}, \ldots, b_m] \, = \, \frac{C_i}{D_i} \, . \]
These are called \defn{partial fractions} or \defn{tails of continued fractions}.
Note that for simple CFs we have \ts $\gcd(C_i,D_i)=1$, but this does not always hold for GCFs.

Similarly, consider a RGCF \. $[\alpha_1,\ldots, \alpha_m \ts; \ts b_0,\ldots, b_m]$ \.
given by \eqref{eq:RGCF-def}.  Let \ts $\al_i = c_i/d_i$ \ts where \ts $\gcd(c_i,d_i)=1$\ts, \ts $1 \le i \le m$\ts.
Recursively define \. $C_i:=C_i(\al_1,\ldots,\al_m \ts ; \ts b_0,\ldots, b_m)$ \.
and \. $D_i := D_i(\al_1,\ldots,\al_m \ts ; \ts b_0,\ldots, b_m)$ \.
 as follows:
\begin{align*}
	&C_m \ := \  b_m \,, \qquad D_m \ := \  1,\\
	& D_{i} \ := \  d_{i+1} \. \big(C_{i+1} \. + \. (\as(\al_{i+1}) -1 ) \. D_{i+1} \big), \\
    &  C_{i} \ := \  b_{i} \. D_{i}  \. + \. c_{i+1} \.\big(  C_{i+1} \. + \. \as(\al_{i+1}) \. D_{i+1} \big).
\end{align*}
It is  easy to see by induction that
\[ [\al_{i+1},\al_{i+2}, \ldots, \al_m \ts ; \ts  b_i, b_{i+1}, \ldots, b_m] \, = \, \frac{C_i}{D_i} \, . \]
The ratios \ts $\frac{C_i}{D_i}$ \ts are called \defn{partial fractions} \ts in this case.
%\newpage

%%%%%%%%%%%%%%%%%%%%%%%%%%%%%%%%%%%%%%%%%%%%%%%%%%%%%%%%%%%%%%%%%%%%%%%%%%%%%%%%%%%%%%%%%%%%%%%%%
%%%%%%%%%%%%%%%%%%%%%%%%%%%%%%%%%%%%%%%%%%%%%%%%%%%%%%%%%%%%%%%%%%%%%%%%%%%%%%%%%%%%%%%%%%%%%%%%%
%%%%%%%%%%%%%%%%%%%%%%%%%%%%%%%%%%%%%%%%%%%%%%%%%%%%%%%%%%%%%%%%%%%%%%%%%%%%%%%%%%%%%%%%%%%%%%%%%
%%%%%%%%%%%%%%%%%%%%%%%%%%%%%%%%%%%%%%%%%%%%%%%%%%%%%%%%%%%%%%%%%%%%%%%%%%%%%%%%%%%%%%%%%%%%%%%%%
%%%%%%%%%%%%%%%%%%%%%%%%%%%%%%%%%%%%%%%%%%%%%%%%%%%%%%%%%%%%%%%%%%%%%%%%%%%%%%%%%%%%%%%%%%%%%%%%%
%%%%%%%%%%%%%%%%%%%%%%%%%%%%%%%%%%%%%%%%%%%%%%%%%%%%%%%%%%%%%%%%%%%%%%%%%%%%%%%%%%%%%%%%%%%%%%%%%
%%%%%%%%%%%%%%%%%%%%%%%%%%%%%%%%%%%%%%%%%%%%%%%%%%%%%%%%%%%%%%%%%%%%%%%%%%%%%%%%%%%%%%%%%%%%%%%%%
%%%%%%%%%%%%%%%%%%%%%%%%%%%%%%%%%%%%%%%%%%%%%%%%%%%%%%%%%%%%%%%%%%%%%%%%%%%%%%%%%%%%%%%%%%%%%%%%%
%%%%%%%%%%%%%%%%%%%%%%%%%%%%%%%%%%%%%%%%%%%%%%%%%%%%%%%%%%%%%%%%%%%%%%%%%%%%%%%%%%%%%%%%%%%%%%%%%

\medskip

\section{Recursive constructions}\label{s:rec}

\subsection{Hybrid sums}\label{ss:rec-def}
Let \ts $P=(X,\prec)$ \ts and \ts $Q=(Y,\prec')$ \ts be posets on \ts $|X|=n$ \ts
and \ts $|Y|=n'$ \ts elements.   Fix \ts $x \in \min(P)$.
The \defnb{hybrid sum} \. $Q \ole_x P$ \.  is the poset \. $R=(X \cup Y,\prec^\di)$ \.
given by the relations
\begin{align*}
u \. \prec^\di \. u' \quad & \text{ for every } \quad u\. \prec \. u', \ \ u,\ts u' \in  X, \\
v \. \prec^\di \. v' \quad & \text{ for every } \quad v\. \prec' \. v', \ \ v,\ts v' \in Y, \\
v \. \prec^\di  \.  u  \quad & \text{ for every } \quad  u \in X-x, \ \ v \in Y.
\end{align*}
Note that $x$ is incomparable to \. $Y$ \. in \ts $R$, and thus $x\in \min(R)$.

We have:
\begin{equation*}
	e(Q \ole_x P) \ = \  e(Q) \. e(P) \, + \,  e(Q\oplus x) \. e(P-x)  \, -  \, e(Q)  \. e(P-x).
\end{equation*}
Indeed, the term \. $e(Q) \. e(P)$ \. counts  linear extensions \ts $f\in \Ec(R)$ \ts for
which \. $f(x) \geq n'+1$. Similarly, the term  \. $e(Q\oplus x)   \. e(P-x)$ \. counts
\ts $f\in \Ec(R)$ \ts for which \. $f(x) \leq n'+1$.  Finally, the term \.
$e(Q)  \. e(P-x)$ \. counts \ts $f\in \Ec(R)$ \ts  for which \. $f(x) = n'+1$.
Because \. $e(Q \oplus x)  = (n'+1) \cdot e(Q)$, we then have
\begin{equation}\label{eq:hybrid-1}
	e(Q \ole_x P) \ = \  e(Q) \. e(P) \ + \  n' \cdot e(Q) \. e(P-x).
\end{equation}
%\com{SH}{The equation \eqref{eq:hybrid-1} is the most important thing to this approach, and everything will fail if it falls. Do you mind double-checking if this is correct? Thanks!}
%
%\smallskip
%
%We also have
%\begin{equation}\label{eq:hybrid-2}
%	e\big((Q \ole_x P)-x \big) \ = \  e(Q) \. e(P-x),
%\end{equation}
%because \. $(Q \ole_x P)-x$ \.  is the series sums of \. $Q$\., \.$\{z\}$\., and \. $P-x$.
%
%\smallskip

It then follows that for all \ts $y \in \min(Q)$, we have
\begin{equation}\label{eq:hybrid-2}
\begin{split}
	e(R-y)  \, = \, e\big((Q -y)\. \ole_x P \big) \,
	  = \,   e(Q-y) \. e(P) \. + \,  (n'-1) \cdot e(Q-y) \. e(P-x).
\end{split}
\end{equation}
%because the term \. $e(Q-y) \. e(P)$ \. counts linear extensions $f$ for which
%$f(x) > f(z)$, while \. $ e\big((Q-y)\oplus x\big) \. e(P-x)$ \. counts
%$f$ 's for which
%$f(x) < f(z)$.
Since \. $(Q \ole_x P)-x \. = \. Q \ole (P-x)$,  we also have:
\begin{equation}\label{eq:hybrid-3}
	e\big((Q \ole_x P)-x \big) \, = \,  e(Q) \. e(P-x).
\end{equation}

Finally, note that
	\begin{equation}\label{eq:width}
	\begin{split}
	  \width(Q \ole_x P) \ \leq  \   \max\big\{\. \width(P)-1, \. \width(Q) \. \big\} \. + \. 1.
	\end{split}
\end{equation}

\begin{rem}
Hybrid sum is a special case of the \defnb{quasi-series composition} \ts
defined similarly in \cite{HJ85} for general subsets of minimal elements.
%, where  \ts $P_1\gets Q$, \ts $A\gets Y$\ts, \ts $P_2 \gets P$\ts, \ts $B\gets X-\{x\}$.
Also, when \ts $Y=\{y\}$, we have \ts $R=\{y\} \oplus_{y,x}  P$,
where \ts $\oplus_{y,x}$ \ts is the direct sum operation defined in \cite[$\S$2]{KS21}.
\end{rem}
\smallskip

\subsection{Properties of hybrid sums}\label{ss:rec-prop}
We now use hybrid sums to construct posets for which the numbers of
linear extensions satisfy recurrence relations emulating continued
fractions.

\smallskip

\begin{lemma}\label{lem:rec-gen}
Let \ts $P=(X,\prec)$ \ts and \ts $Q=(Y,\prec')$ \ts be posets on \ts $n=|X|$ \ts
and \ts $n'=|Y|$ \ts elements, and let \ts
$x \in \min(P)$, \. $y \in \min(Q)$.
Let \ts $R=(Z,\prec^\di)$ \ts be a poset and let \ts $z \in \min(R)$ be given by 
\[R \ := \  Q \ole_x P, \qquad  z \ := \ y. \]
Then we have 
\begin{align*}
		e(R) \ &= \  e(Q) \. \big( e(P) \ + \  n' \cdot  e(P-x)\big),  \\
		e(R-z) \ &= \  e(Q-y) \. \big( e(P) \ + \  (n'-1) \cdot  e(P-x)\big),\\
		 |Z| \ &= \ n \. + \. n',  \\
		   \width(R) \ &\leq \  \max\big\{\ts\width(P),\. \width(Q)+1\ts\big\}.
	\end{align*}
Additionally, we have:
	\[ \agr(R,z) \ = \ \agr(Q,y) \bigg( 1 \. + \. \frac{1}{n'-1 \. +  \. \agr(P,x)  }\bigg). \]

\end{lemma}

\smallskip

\begin{proof}
%	Let \. $R:= Q \ole_x P$, and let \ts $z:=y$.
	The first four conclusions follows from \eqref{eq:hybrid-1}, \eqref{eq:hybrid-2}, \eqref{eq:width}.
We conclude that
	\begin{align*}
		\ag(R,z) \ &= \ \frac{e(Q)}{e(Q-y)} \, \cdot \, \frac{e(P) \. + \. n' \cdot e(P-x)}{e(P) \. + \. (n'-1) \cdot e(P-x)} \\
		&= \ \ag(Q,y) \. \bigg(1\. + \. \frac{e(P-x)}{e(P) \. + \. (n'-1) \cdot e(P-x)} \bigg) \\
		&= \ \ag(Q,y) \. \bigg( 1 \. + \. \frac{1}{\ag(P,x) \. + \. (n'-1) }\bigg),
	\end{align*}
as desired.
\end{proof}

\smallskip

\begin{lemma}\label{lem:rec-add}
Let \ts $P=(X,\prec)$ \ts be a poset on \ts $|X|=n$ \ts elements,  let \ts
$x \in \min(P)$, and let \ts $b \geq 0$.
Let \ts $R=(Z,\prec^\di)$ \ts be a poset and let \ts $z \in \min(R)$ be given by 
\[R \ := \  C_b \ole_x P, \qquad  z \ := \ x, \]
where $C_b$ is a chain of $b$ elements.
Then we have
	\begin{align*}
		e(R) \ &= \ e(P) \. + \. b\cdot e(P-x),\\
		e(R-z) \ &= \ e(P-x),\\
		|Z| \ &= \ n \. + \. b,\\
	\width(R) \ &\leq \  \max\big\{\ts \width(P), \ts 2\ts \big\}.
	\end{align*}
Additionally, we have:
	\[ \agr(R,z) \ = \  b \. + \, \agr(P,x). \]
%	, and
%	\[ \gcd(e(R),e(R-z)) \ = \  \gcd(e(P), e(P-x)).  \]
\end{lemma}

\smallskip

\begin{proof}
%	Let $Q:=C_b$ be a chain of $b$ elements, so \ts $e(Q)=1$.
%Let \. $R:= Q \ole_x P$, and let \. $z:=x$.
The first four conclusions follows from  \eqref{eq:hybrid-1}, \eqref{eq:hybrid-3},
and \eqref{eq:width}. We conclude that
\[ \ag(R,z) \ = \   \frac{e(P) \. + \.  b \cdot e(P-x)}{ e(P-x)} \ = \ \ag(P,x) \. + \. b, \]
as desired.
\end{proof}

\smallskip

By combining the two lemmas above, we get the following:

\smallskip

\begin{lemma}
\label{l:most-general}
Let \ts $P=(X,\prec)$ \ts and \ts $Q=(Y,\prec')$ \ts be posets on \ts $n=|X|$ \ts
and \ts $n'=|Y|$ \ts elements, and let \ts
$x \in \min(P)$, \. $y \in \min(Q)$.  Fix \ts $b \geq 0$.
Let \ts $R=(Z,\prec^\di)$ \ts be a poset and let \ts $z \in \min(R)$ be given by 
\[ R \ := \  C_b \ole_y (Q \ole_x P), \qquad z \ := \ y. \]
where $C_b$ is a chain of $b$ elements.
Then we have
	\begin{align*}
	e(R-z) \ &= \  e(Q-y) \. \big( e(P) \ + \  (n'-1) \.  e(P-x)\big),\\
		e(R) \ &= \   b \cdot e(R-z) \ + \  e(Q) \. \big[\ts e(P) \. + \.  n' \cdot  e(P-x)\ts\big],  \\
	|Z| \ &= \ n \. + \. n' \. + \. b,  \\
	\width(R) \ &\leq \  \max\big\{\ts \width(P),\. \width(Q)+1,2\ts \big\}.
\end{align*}
Additionally, we have:
	\[ \agr(R,z) \ = \ b \ + \ \agr(Q,y) \bigg( 1+ \frac{1}{n'-1 \. +  \. \agr(P,x)  }\bigg). \]
\end{lemma}

\smallskip

\begin{proof}
	This follows from first applying Lemma~\ref{lem:rec-gen} then applying Lemma~\ref{lem:rec-add}.
\end{proof}

\smallskip

\begin{lemma}\label{lem:rec-CF}
Let \ts $P=(X,\prec)$ \ts be a poset on \ts $|X|=n$ \ts elements,  let \ts
$x \in \min(P)$, and let \ts $b \geq a \ge 0$.
Let \ts $R=(Z,\prec^\di)$ \ts be a poset and let \ts $z \in \min(R)$ be given by 
\[ R \ := \  C_{b-a} \ole_y  \big( (y \oplus C_{a-1}) \ole_x P \big), \qquad z \ := \ y, \]
where $C_b$ is a chain of $b$ elements, and $y$ is an element not contained in $P$.
Then we have
%Then there exists a poset \ts $R=(Z,\prec^\di)$ \ts and \ts $z \in \min(R)$,
%such that
		\begin{align*}
		e(R-z) \ &= \  e(P) \ + \  (a-1) \cdot  e(P-x),\\
		e(R) \ &= \   (b-a) \cdot e(R-z) \ + \ a \cdot \big[\ts e(P) \. + \.   a\cdot  e(P-x)\ts\big] \ = \ b \cdot e(R-z) \. + \.a \cdot e(P-x),  \\
		|Z| \ &= \ n \. + \. b,  \\
		\width(R) \ &\leq \  \max\big\{\ts \width(P),\ts 3\big\}.
	\end{align*}
Additionally, we have:
	\[ \agr(R,z) \ = \  b \. + \, \frac{a}{a-1 + \agr(P,x)}\,. \]
\end{lemma}

\smallskip

\begin{proof}
	Let \. $Q=(Y,\prec') :=y \oplus C_{a-1}$\..
	Note that
	\[ e(Q) \. = \.  a, \quad e(Q-y) \. = \. 1,  \quad |Y| \. = \. a, \quad \text{and} \quad \width(Q)\. = \.2. \]
	The lemma now follow from  substituting  \. $b \gets (b-a)$ \. into Lemma~\ref{l:most-general}.
\end{proof}

\smallskip

\subsection{A flip-flop construction}  \label{ss:rec-another}
We will need the following variation on the hybrid sum construction
to prove Theorem~\ref{t:CP-relative}.
%\com{SH}{Out of curiosity, why is it called flip-flop?}

Let \ts $P=(X,\prec)$ \ts and \ts $Q=(Y,\prec')$ \ts be posets on \ts $n=|X|$ \ts
and \ts $n'=|Y|$ \ts elements, and let \ts $x \in \min(P)$, \. $y \in \min(Q)$.
The \defnb{flip-flop poset} 
 \. $R=(Z, \prec^\di)$ \. is defined by 
\[  Z \  := \   (X-x) \. \cup \.   (Y-y)  \. \cup \{z,v\},  \]
where \ts $z,v$ are new elements.  The partial order \ts $\prec^\di$ \ts
is  defined by
\begin{align*}
	p \. \prec^\di \. p' \quad & \text{ for every } \quad  p,\ts p' \in  X-x \quad \text{s.t.} \quad p\. \succ \. p', \\
	q \. \prec^\di \. q' \quad & \text{ for every } \quad  q,\ts q' \in Y-y \quad \text{s.t.} \quad q\. \prec' \. q', \\
	p  \. \prec^\di \. z \quad & \text{ for every } \quad \. p \in X-x \quad \text{s.t.} \quad x \. \prec \. p\ts,\\
	z  \. \prec^\di \. q \quad & \text{ for every } \quad \. q \in Y-y \quad \text{s.t.} \quad 	y \. \prec' \. q\ts,\\
	p \. \prec^\di \. v \. \prec^\di \. q  \quad & \text{ for every } \quad p \in X-x, \quad q \in Y -y\ts, \\
	& \quad \text{ and }  \quad z \, ||_{\prec^\di} \, v\ts.
\end{align*}

\begin{lemma}\label{l:rec-relative}
The flip-flop poset \. $R=(Z, \prec^\di)$ \. satisfies
%Then there exists a poset \ts $R=(Z,\prec^\di)$ \ts and \ts $z \in Z$,
%such that
	\begin{align*}
		e(R) \ &= \  e(P) \. e(Q-y) \. + \. e(P-x) \. e(Q), \\
		e(R-z) \ &= \ e(P-x) \. e(Q-y),\\
		|Z| \ &= \ n \. + \. n',\\
		\width(R) \ &\leq  \ \width(P) \. + \. \width(Q).
	\end{align*}
Additionally, we have:
	\[ \agr(R,z) \, = \, \agr(P,x) \. + \. \agr(Q,y).
    \]
%\frac{e(Q)}{e(Q-y)} \ = \    \frac{e(P)}{e(P-x)} \.  + \.  \frac{e(P')}{e(P'-x')}.
\end{lemma}

We warn the reader that the element \ts $z$ \ts is not necessarily a minimal element
of~$R$, so this construction cannot be easily iterated.

\smallskip

\begin{proof}
We have:
\[	
e(R) \ = \  e(P) \cdot e(Q-y) \. + \. e(P-x) \cdot e(Q).
\]
Indeed, the factor \. $e(P) \cdot e(Q-y)$ \. counts linear extensions \ts
$f\in \Ec(R)$ \ts for which \ts $f(u)<f(v)$, while the factor \. $e(P-x) \. e(Q)$ \.
counts linear extensions \ts $f\in \Ec(R)$ \ts for which \ts $f(u)>f(v)$.
	Also note that
	\[ e(R-z) \ = \  e(P-x) \cdot e(Q-y),  \]
	because \. $(R-z)$ \. is isomorphic to the linear sum  \. $(P-x) \ole \{v\} \ole (Q-y)$.
Finally, note that	\begin{align*}
		\width(R) \ &\leq  \ \width(P) \. + \. \width(Q),\\
		|Z| \ &= \ |X| \. + \. |Y|,
	\end{align*}
	by construction.  This completes the proof.
\end{proof}

%%%%%%%%%%%%%%%%%%%%%%%%%%%%%%%%%%%%%%%%%%%%%%%%%%%%%%%%%%%%%%%%%%%%%%%%%%%%%%%%%%%%%%%%%%%%%%%%
%%%%%%%%%%%%%%%%%%%%%%%%%%%%%%%%%%%%%%%%%%%%%%%%%%%%%%%%%%%%%%%%%%%%%%%%%%%%%%%%%%%%%%%%%%%%%%%%
%%%%%%%%%%%%%%%%%%%%%%%%%%%%%%%%%%%%%%%%%%%%%%%%%%%%%%%%%%%%%%%%%%%%%%%%%%%%%%%%%%%%%%%%%%%%%%%%
%%%%%%%%%%%%%%%%%%%%%%%%%%%%%%%%%%%%%%%%%%%%%%%%%%%%%%%%%%%%%%%%%%%%%%%%%%%%%%%%%%%%%%%%%%%%%%%%
%%%%%%%%%%%%%%%%%%%%%%%%%%%%%%%%%%%%%%%%%%%%%%%%%%%%%%%%%%%%%%%%%%%%%%%%%%%%%%%%%%%%%%%%%%%%%%%%
%%%%%%%%%%%%%%%%%%%%%%%%%%%%%%%%%%%%%%%%%%%%%%%%%%%%%%%%%%%%%%%%%%%%%%%%%%%%%%%%%%%%%%%%%%%%%%%%
%%%%%%%%%%%%%%%%%%%%%%%%%%%%%%%%%%%%%%%%%%%%%%%%%%%%%%%%%%%%%%%%%%%%%%%%%%%%%%%%%%%%%%%%%%%%%%%%
%%%%%%%%%%%%%%%%%%%%%%%%%%%%%%%%%%%%%%%%%%%%%%%%%%%%%%%%%%%%%%%%%%%%%%%%%%%%%%%%%%%%%%%%%%%%%%%%

\medskip

\section{Proofs} \label{s:proofs}

\subsection{Proof of Theorem~\ref{t:GCF}}
We prove the claim by induction on $m$.  First, let $m=0$.  Recall the notation in~$\S$\ref{ss:def-CF}.
Note that condition~\eqref{eq:GCF-balanced} becomes \. $b_0 \geq 1$, which holds by the assumption.
Let \. $P=(X,\prec):=\{x\} \oplus C_{b_0-1}$.  Then we have:
\[ e(P) \, = \,  b_0 \, = \,  C_0(b_0) \quad \text{and}  \quad e(P-x)  \, =  \, 1 \, = \, D_0(b_0).
\]
We also have \. $|X|=b_0$ \. and \. $\width(P)=2$, as desired.
	
Suppose now that the claim holds for \ts $m-1$.  Let \. $b_1':=b_1-a_1+1$.
The balanced assumptions \eqref{eq:GCF-balanced} gives \. $b_1'\geq a_2$.
Thus, by the inductive assumption,
there exist \ts $P'=(X',\prec')$ \ts and $x' \in \min(P)'$,  	
such that
\begin{align*}
	e(P'-x') \ &= \ D_0(a_2,\ldots, a_m \ts ; \ts b_1', b_2, \ldots, b_m) 	\ = \ D_1(a_1,\ldots, a_m \ts ; \ts b_0, b_1, \ldots, b_m), \\
e(P') \ &= \   C_0(a_2,\ldots, a_m \ts ; \ts  b_1', b_2, \ldots, b_m) \\
\ &= \ C_0(a_2,\ldots, a_m \ts ; \ts  b_1, b_2, \ldots, b_m) \. - \. (a_1-1) \cdot D_0(a_2,\ldots, a_m \ts ; \ts  b_1, b_2, \ldots, b_m) \\
\ &= \ C_1(a_1,\ldots, a_m \ts ; \ts  b_0, \ldots, b_m) \. - \. (a_1-1) \cdot D_1(a_1,\ldots, a_m \ts ; \ts  b_0, \ldots, b_m).
\end{align*}
%(Note  that \. $b_1'$ \. here satisfies \. $b_1'\geq a_2$ \. by condition \eqref{eq:GCF-balanced}, so this induction step is allowed).
Now, apply  Lemma~\ref{lem:rec-CF} to \ts $P \gets P'$ \ts with \. $b \gets b_0$ \. and \. $a \gets a_1\ts$ \. and \. $x \gets x'$.
We obtain a poset \ts $R=(Z,\prec^\di)$ \ts on \ts $|Z|=n$ \ts elements, and \ts $z \in \min(R)$, such that
\begin{align*}
	e(R-z)  \ &= \ e(P') \. + \. (a_1-1) \cdot  e(P'-x') \ = \  C_1(a_1,\ldots, a_m \ts ; \ts   b_0, \ldots, b_m)\\
	&= \  D_0(a_1,\ldots, a_m \ts ; \ts   b_0, \ldots, b_m),\\
	e(R) \ &= \ b_0 \cdot e(R-z) \. + \. a_1 \cdot e(P'-x') \\
& = \ b_0 \cdot  D_0(a_1,\ldots, a_m \ts ; \ts   b_0, \ldots, b_m) \. + \. a_1 \cdot D_1(a_1,\ldots, a_m \ts ; \ts   b_0, \ldots, b_m) \\
	&= \  C_0(a_1,\ldots, a_m  \ts ; \ts   b_0, \ldots, b_m).
\end{align*}
Since induction assumption implies that $|X|'= \aG(a_2,\ldots,a_m \ts ; \ts b_1',b_2,\ldots,b_m)$, 
we also have
\begin{align*}
		n \ &= \  b_0 \. + \. |X'| \, = \,  b_0 \. + \. b_1' \. + \. \sum_{i=2}^m \ts b_i \. - \.
\sum_{i=2}^m \ts a_i \. + \, m -1 \\
		\ &  = \ \sum_{i=0}^m \ts b_i \. - \. \sum_{i=1}^m \ts a_i \. + \, m  \ = \  \aG(a_1,\ldots,a_m \ts ; \ts b_0,\ldots,b_m),
 \end{align*}
and \. $\width(R) \. \leq \. \max\big\{\ts \width(P'),\. 3\ts\big\} \. \leq  \. 3$.  Finally, we have:
\begin{align*}
	\ag(R,z) \ = \  \cfrac{ C_0(a_1,\ldots, a_m  \ts ; \ts   b_0, \ldots, b_m)}{ D_0(a_1,\ldots, a_m  \ts ; \ts   b_0, \ldots, b_m)} \ = \ [a_1,\ldots, a_m \ts ; \ts b_0,\ldots, b_m].
\end{align*}
The theorem is now complete by substituting \ts $P \gets R$ \ts and \ts $x \gets z$. \qed

\smallskip

\subsection{Proof of Theorem~\ref{t:RGCF} }
We prove the claim by induction on~$m$.
% The case $k=0$ is exactly Theorem~\ref{t:GCF}.
% So we assume that the claim is true for $k-1$.
For $m=0$, let \. $P = (X,\prec) :=x \oplus C_{b_0-1}\ts$.
We have:
\[ e(P) \, = \, b_0 \, = \  C_0(b_0) \quad \text{and} \quad e(P-x)  \, =  \, 1 \, = \, D_0(b_0). \]
We also have \. $|X|=b_0$ \. and \. $\width(P)=2$, which proves the case \ts $m=0$.

We now suppose the claim is already proved for \ts $(m-1)$. By the induction assumption,
there exists a poset \ts $P'=(X',\prec')$ \ts and element \ts $x' \in \min(P')$, such that
\begin{align*}
	e(P'-x') \ &= \ D_0(\al_2,\ldots, \al_m \ts ; \ts  b_1, b_2, \ldots, b_m)
	\ = \ D_1(\al_1,\ldots, \al_m \ts ; \ts b_0, b_1, \ldots, b_m), \\
	e(P') \ &= \   C_0(\al_2,\ldots, \al_m \ts ; \ts  b_1, b_2, \ldots, b_m)
	\ = \ C_1(\al_1,\ldots, \al_m \ts ; \ts  b_0, \ldots, b_m).
\end{align*}
Applying Theorem~\ref{t:KS-CF} to \ts $\al_1$,
there exists a poset \ts $Q=(Y,\prec')$ \ts and \ts $y \in \min(Q)$ \ts such that
\begin{align*}
	&e(Q) \, = \, c_1, \quad e(Q-y) \, = \, d_1\ts, \quad |Y| \, = \,  \as(\alpha_1) \quad \text{and} \quad \width(Q) \, \leq \, 2.
\end{align*}

Now, apply  Lemma~\ref{l:most-general} to posets \ts $P'$, \ts $Q$,  and element \. $b_0$.
We obtain a poset $P=(X,\prec)$ \ts and \ts $x \in \min(P)$, such that
\begin{align*}
	e(P-x) \ &= \  e(Q-y) \. \big[ e(P') \ + \  (|Y|-1) \.  e(P'-x)\big],\\
	&= \  d_1\cdot\big[C_1(\al_1,\ldots, \al_m \ts ; \ts b_0, \ldots, b_m) \. + \. \big( \as(\alpha_1)-1\big) \cdot D_1(\al_1,\ldots, \al_m \ts ; \ts  b_0, b_1, \ldots, b_m) \big] \\
	&= \ D_0(\al_1,\ldots, \al_m \ts ; \ts b_0, \ldots, b_m).
\end{align*}
We also have
\begin{align*}
	 e(P) \ &= \ b_0 \cdot e(P-x) \ + \  e(Q) \. \big[ e(P') \ + \  |Y| \cdot  e(P'-x')\big] \\
	&= \  b_0 \cdot  D_0(\al_1,\ldots, \al_m\ts ; \ts  b_0, \ldots, b_m) \\
	& \qquad \.  + \. c_1 \cdot \big[ C_1(\al_1,\ldots, \al_m \ts ; \ts b_0, \ldots, b_m) \ + \
\as(\alpha_1)  \cdot  D_1(\al_1,\ldots, \al_m\ts ; \ts  b_0, b_1, \ldots, b_m)\big] \\
	&= \ C_0(\al_1,\ldots, \al_m \ts ; \ts b_0, \ldots, b_m)
\end{align*}
and
\[ \width(P) \, \leq \,  \max\big\{\ts \width(P'),\ts \width(Q)+1,\ts 2\ts \big\} \, \leq \, 3. \]
Finally, we have
\begin{align*}
 |X| \ &= \  |X'|\. + \. |Y| \. + \. b_0 \\
  &= \ (b_1+\ldots+ b_m)  \. + \. \as(\al_2)+ \ldots + \as(\al_m))  \. + \. \as(\al_1) \.  + \. b_0 \\
  &= \  \aR(\al_1,\ldots,\al_m \ts ; \ts b_0,\ldots,b_m).
\end{align*}
This completes the proof. \qed

\smallskip

\subsection{Proof of Propositions~\ref{p:GCF-imply} and~\ref{p:RGCF-imply}}
For Proposition~\ref{p:GCF-imply}, recall from the introduction
that the Conjecture~\ref{conj:gen-asy}
implies Conjecture~\ref{conj:KS-main} for prime~$d$. Indeed,
by Theorem~\ref{t:KS-CF}  for a GCF \.
$[a_1,\ldots, a_m \ts ; \ts b_0,\ldots, b_m]=\frac{d}{c}\ts$,
we obtain a poset \ts $P=(X,\prec)$ \ts and \ts $x \in X$ \ts  such that \ts
$|X|=\agg\big(\frac{d}{c}\big) \le C \ts \log d$ \ts and \ts
\ts $\frac{e(P)}{e(P-x)}=\frac{d}{c}$\ts.
By the reduced condition on the definition of $\agg$, it then follows that
 $e(P)=d$,
as desired. %\com{SH}{Should we add here by the reduced condition in $\agg$?}

To show that the first part of Conjecture~\ref{conj:gen-asy} suffices,
let \. $p_1^{m_1}\ldots p_\ell^{m_\ell}$ \.
be the prime factorization of~$d$.  For each prime \ts $p_i$,
let \ts $P_i=(X_i, \prec_i)$ \ts be the corresponding poset with \ts $e(P_i)=p_i$ \ts
and \ts $|X_i| \le C \ts \log p_i$.  Define
\[  P  \ := \  \underbrace{P_1 \ole \cdots \ole P_1}_{m_1\text{ times}} \,  \ole \,
\cdots \, \ole\,   \underbrace{P_\ell \ole \cdots \ole P_\ell}_{m_\ell\text{ times}}
\]
be the linear sum of posets~$P_i$.  We have:
\[e(P) \,  = \,  \prod_{i=1}^\ell \. e(P_i)^{m_i} \, = \, d\ts,
\]
and
\begin{align*}
	|X|  \, = \, \sum_{i=1}^\ell \. m_i \.\ts |X_i| \ \leq \ C \. \sum_{i=1}^\ell \. m_i \. \log p_i \ = \, C \ts \log d\ts.
\end{align*}
This completes the proof of Proposition~\ref{p:GCF-imply}.
The proof of Proposition~\ref{p:RGCF-imply} follows verbatim.\qed

\smallskip

\subsection{Proof of Theorem~\ref{t:CP-relative} }
First, observe that there exists a constant \ts $C>0$, such that
for all coprime integers \ts $a,\ts b\le d$ \ts which satisfy \. $C < b \leq a \leq 2b$,
there exists a positive integer \. $\ell:=\ell(a,b)$ \. such that \ts $1\le \ell <  b$, and
\begin{equation}\label{eq:ell-ab}
\as\big(\tfrac{\ell}{b}\big) \, \leq \,  2 \. \log b \. \log \log b  \quad \text{ and } \quad
\as\big(\tfrac{a-\ell}{b}\big) \, \leq \,  2 \. \log b \. \log \log b\ts.
\end{equation}
Indeed, by Theorem~\ref{t:Ruka} and using \. $\frac{12}{\pi^2}<2$, for random \ts $\ell \in \{1,\ldots,b\}$,
the probability that each inequality fails \. $\to 0$ \.
as \. $b\to \infty$.  %Thus the number of such \ts $\ell$ \ts is $n\big(1-o(1)\big)$.
Taking \ts $C$ \ts large enough so that each probability is \ts $<\frac12$ \ts proves the claim.

Let \ts $a,b$ \ts be given by
\[ a \, := \,  c \. + \. d \.  - \, \left\lfloor \frac{d}{c} \right \rfloor \. c, \quad b  \. := \. c\ts,
\]
so that \. $b \leq a \leq 2\ts b$ \. and \. $b \leq d$\..
From above, there exists \. $1\le \ell \le b$, such that \eqref{eq:ell-ab} holds.
Let
\[ \alpha \, := \,  1 \. + \. \frac{\ell}{b} \quad \text{and}
\quad \be \, := \,  \left\lfloor \frac{d}{c} \right \rfloor \.  - \. 2  \. + \.  \frac{a-\ell}{b}\..
\]
It follows from the construction that \. $\alpha+ \be =\frac{d}{c}$\..
Since \. $\frac{d}{c}\geq 3$, we have \. $\alpha,\be \geq 1$.

Applying Theorem~\ref{t:GCF} to simple continued fractions, we obtain a poset \ts $P=(X,\prec)$ \ts
and element \ts $x \in \min(P)$, such that
\[
    \ag(P,x) \, = \,  \alpha \quad \text{and} \quad |X| \, = \, 1 \. + \. \as\big(\tfrac{\ell}{b}\big).
\]
Similarly, we obtain a poset \ts $Q=(Y,\prec')$ \ts
and element \ts $y \in \min(Q)$, such that
\[
    \ag(Q,y) \, = \,  \be \quad \text{and} \quad |Y| \, = \, \left\lfloor \tfrac{d}{c} \right \rfloor \.  - \. 2  \. + \.   \as\big(\tfrac{a-\ell}{b}\big).
\]
By Lemma~\ref{l:rec-relative}, there exists a poset \. $R=(Z,\prec^\di)$ \ts and
element \. $z \in Z$, such that	
\[ \ag(R,z) \, = \,   \ag(P,x) \.  + \.  \ag(Q,y) \, = \,  \frac{d}{c}\,,
\]
	and
\[  |Z| \ = \  |X| \. + \. |Y| \ =  \ \left\lfloor \frac{d}{c} \right \rfloor \. - \. 1 \. + \.  \as\big(\tfrac{\ell}{b}\big) \.  + \.  \as\big(\tfrac{a-\ell}{b}\big) \, \leq \, \frac{d}{c} \.  + \. O(\log d \. \log \log d).
\]
This completes the proof. \qed

%\vfil

\medskip

\section{Final remarks and open problems} \label{s:finrem}

\subsection{} \label{ss:finrem-nature}
The nature of connections between counting combinatorial objects
and continued fractions described in~$\S$\ref{ss:intro-for}
is clear and easy to explain: when objects are decomposed into smaller
objects, they often have simple recurrences of the type described
in~$\S$\ref{ss:def-CF}.  Fundamentally, this is the same reason
why the generating functions are so powerful in combinatorial enumeration,
see e.g.\ \cite{GJ83,Sta-EC}.
And yet, every time such a connection is found it is an unexpected delight,
stemming both from the sheer elegance of continued fractions as well as
the power of technical tools developed for them.  While we tend to be
swayed by the latter arguments, we appreciate the former sentiments.

\subsection{} \label{ss:finrem-hist}
The upper bound in Larcher's Theorem~\ref{t:Larcher}
was sharpened by Rukavishnikova \cite{Ruka} to \ts $O(\log d \ts \log \log d)$.
Since \ts $\frac{d}{\phi(d)}$ \ts can be as large as \ts $C\log \log d$,
see e.g.\ \cite[Thm~328]{HardyWright}, this is a significant asymptotic improvement.
This result was further sharpened by Aistleitner, Borda and Hauke \cite[Cor.~2]{ABH23},
who proved that for all \ts $d\ge 3$ \ts there exist \ts $1\le c < d$, such that
\begin{equation}\label{eq:ABH-cor}
\as\big(\tfrac{c}{d}\big) \, \le  \, \tfrac{12}{\pi^2} \. \log d \. \log \log d \. + \. O\big(\log d)\ts.
\end{equation}
Note that we are using only prime \ts $d$ \ts for our applications, which it why
we postponed this recent result.  We note in passing that the authors of \cite{KS21}
stated Conjecture~\ref{conj:KS-asy} in the generality of all~$d$; while plausible
this remains out of reach with the existing technology.  They were unaware of the
earlier work and rediscovered Theorem~\ref{t:Larcher}.\footnote{Personal communication.}

\subsection{} \label{ss:finrem-hist}
The asymptotics in the upper bound \eqref{eq:ABH-cor} cannot be easily improved by
probabilistic arguments.  This follows from a version on the tail estimates \eqref{eq:Ruka}
given in \cite{Ruk06}.  A stronger result was proved in~\cite[Thm~1]{ABH23},
which implies that for all \ts $C<0$, \ts $d\ge 3$, and \ts $\ve=\ve(d) > (\log d)^C$,
for the \ts $(1-\ve)$ \ts fraction of \ts $c\in \{1,\ldots,d\}$ \ts with \ts $\gcd(c,d)=1$, we have:
\begin{equation}\label{eq:ABH-tail}
\left|\.\as\big(\tfrac{c}{d}\big) \, - \, \tfrac{12}{\pi^2} \. \log d \. \log \log d\.\right|
\ = \  O\big(\tfrac{\log d}{\ve}\big).
\end{equation}
Of course, this does not preclude the outlying small values predicted by
Zaremba's conjecture.  In fact, as was pointed out in~\cite{ABH23},
the distribution of \. $\as\big(\frac{c}{d}\big)$ \ts is heavy-tailed
and has a large mean:
\begin{equation}\label{eq:YK}
    \tfrac{1}{\phi(d)} \. \sum_c \.  \asr\big(\tfrac{c}{d}\big) \ = \
    \tfrac{6}{\pi^2} \. (\log d)^2 \. + \. O\big((\log d) (\log \log d)^2\big), %\quad \text{as \ \  $n\to \infty$},
\end{equation}
where the summation is over all \. $c\in \{1,\ldots,d\}$ \. such that \. $\gcd(c,d)=1$.
This was proved independently in \cite{Lie83,Pan82,YK75}.

\subsection{} \label{ss:finrem-probab}
It was pointed out by Kravitz and Sah (see Remark~5.31 in~\cite{CP23a}),
that the numerator \ts $c$ \ts in Theorem~\ref{t:KS-CF} can be found in probabilistic
polynomial time \. poly$(\log d)$.  Tail estimates \eqref{eq:ABH-tail}
give a simpler (and faster) probabilistic algorithm: pick a random~$c$,
check if \ts $\gcd(c,d)=1$, compute a simple~CF \eqref{eq:CF-def},
repeat if \. $\as\big(\tfrac{c}{d}\big)> 2 \log d \ts \log \log d$.
% In fact, weaker tail bounds \eqref{eq:Ruka} suffice
% for this argument.
%
It is an interesting open problem if this
can be done deterministically.  More broadly, is there a deterministic polynomial
time construction of a poset with exactly $n$ linear extensions?  So far, the only
deterministic construction we know of is  by Tenner~\cite{Ten}, which is exponential
in \ts $(\log n)$.

\subsection{} \label{ss:finrem-Zar}
Zaremba's Conjecture~\ref{conj:Zaremba} is often stated with \ts $A=5$ \ts or even
\ts $A=4$ \ts for all sufficiently large integers.  It is known to hold for integers of the
form \ts $2^m\ts 3^n$, for other families of powers of small primes and sufficiently
large powers of all primes, where the constant~$A$ can depend on the prime, see \cite{Shu23}.
We refer to \cite[$\S$6.2]{BPSZ14} for an elegant presentation of the \ts $2^m$ \ts case.
Of course, the Kravitz--Sah Conjecture~\ref{conj:KS-main} is trivial in this case.
Note that the constant $50$ in the Bourgain--Kontorovich theorem that was used in the
proof of Corollary~\ref{c:BK}, has been improved to~$5$ in~\cite{Hua15}.
See \cite{Kan21} for further extensions, and \cite[$\S7$]{Shk21}
for an overview.

\subsection{} \label{ss:finrem-count}
It would be interesting to find an elementary proof of the first part
of Corollary~\ref{c:BK}.  The result is especially surprising given
that the bound is obtained on a relatively small family of posets
of width two.  On the other hand, we know of no nontrivial bound
for the much larger family of height two posets (cf.~\cite{Sou23}).

\subsection{} \label{ss:finrem-height-two}
In \cite[Conj.~5.17]{CP23a}, we conjecture that all but finitely many
integers are the numbers of linear extensions of posets of height two.
We also observe (Prop.~5.18, ibid.), that this would imply
Conjecture~\ref{conj:KS-main} with a sharp \. $\Theta\big(\frac{\log n}{\log\log n}\big)$
\. asymptotics.

%\subsection{} \label{ss:finrem-refs}
%

\vskip.6cm
{\small

\subsection*{Acknowledgements}
We are grateful to Christoph Aistleitner,
Milan Haiman, Oleg Karpenkov,  Alex Kontorovich,
Noah Kravitz, Greta Panova, Ashwin Sah, Rolf Schiffler,
Ilya Shkredov, Nikita Shulga and Alan Sokal for interesting 
discussions and helpful remarks.
Special thanks to Maria Rukavishnikova for sending us~\cite{Ruk06}.
Last but not least, we would like to thank the anonymous referees for valuable comments and suggestions.
Both authors were partially supported by the~NSF.
}

%\newpage
 \vskip1.1cm

%%%%%%%%%%%%%%%%%%%%%%%%%%%%%%%%%%%%%%%%%%%%%%%%%%%%%%%%%%%%%%%%%%%%%%%%

%%%%%%%%%%%%%%%%%%%%%%%%%%%%%%%%%%%%%%%%%%%%%%%%%%%%%%%%%%%%%%%%%%%%%%%%%%%%%%%%%%%%%%%%%%%%%
%%%%%%%%%%%%%%%%%%%%%%%%%%%%%%%%%%%%%%%%%%%%%%%%%%%%%%%%%%%%%%%%%%%%%%%%%%%%%%%%%%%%%%%%%%%%%
%%%%%%%%%%%%%%%%%%%%%%%%%%%%%%%%%%%%%%%%%%%%%%%%%%%%%%%%%%%%%%%%%%%%%%%%%%%%%%%%%%%%%%%%%%%%%
%%%%%%%%%%%%%%%%%%%%%%%%%%%%%%%%%%%%%%%%%%%%%%%%%%%%%%%%%%%%%%%%%%%%%%%%%%%%%%%%%%%%%%%%%%%%%
%%%%%%%%%%%%%%%%%%%%%%%%%%%%%%%%%%%%%%%%%%%%%%%%%%%%%%%%%%%%%%%%%%%%%%%%%%%%%%%%%%%%%%%%%%%%%
%%%%%%%%%%%%%%%%%%%%%%%%%%%%%%%%%%%%%%%%%%%%%%%%%%%%%%%%%%%%%%%%%%%%%%%%%%%%%%%%%%%%%%%%%%%%%
%%%%%%%%%%%%%%%%%%%%%%%%%%%%%%%%%%%%%%%%%%%%%%%%%%%%%%%%%%%%%%%%%%%%%%%%%%%%%%%%%%%%%%%%%%%%%
%%%%%%%%%%%%%%%%%%%%%%%%%%%%%%%%%%%%%%%%%%%%%%%%%%%%%%%%%%%%%%%%%%%%%%%%%%%%%%%%%%%%%%%%%%%%%

\end{document}